\documentclass[11pt,a4paper]{amsart}
\usepackage{amsmath,amsthm}
\usepackage{comment}
\newtheorem{thm}{Theorem}[section]
\newtheorem{cor}[thm]{Corollary}
\newtheorem{lem}[thm]{Lemma}
\newtheorem{prop}[thm]{Proposition}

\newtheorem{ques}{\sc Question}

\newtheorem{eks}{\sc Example}
\theoremstyle{definition}
\newtheorem{defn}[thm]{Definition}
\theoremstyle{remark}
\newtheorem{rem}{Remark}[section]
\numberwithin{equation}{section}
\newcommand{\tn}{|\mspace{-1mu}|\mspace{-1mu}|}

\newcommand{\eps}{\varepsilon}
\newcommand{\xast}{x^\ast}
\newcommand{\Xast}{X^{\ast}}
\newcommand{\Xastast}{X^{\ast\ast}}
\DeclareMathOperator{\ext}{ext}
\DeclareMathOperator{\linspan}{span}
\DeclareMathOperator{\dens}{dens}
\begin{document}

\title[On thickness and thinness of Banach spaces]%
{On thickness and thinness of Banach spaces}%
\author[T.~A.~Abrahamsen]{Trond A. Abrahamsen}
\author[J.~Langemets]{Johann Langemets}
\author[V.~Lima]{Vegard Lima}
\author[O.~Nygaard]{Olav Nygaard}

\address{University of Tartu, J. Liivi 2,
50409 Tartu, Estonia.}
\email{johann.langemets@ut.ee}

\address{Aalesund University College, Postboks 1517,
N-6025 {\AA}lesund
Norway.}
\email{Vegard.Lima@gmail.com}

\address{Department of Mathematics, Agder University, Servicebox 422,
4604 Kristiansand, Norway.}
\email{Trond.A.Abrahamsen@uia.no}
\urladdr{http://home.uia.no/trondaa/index.php3}
\email{Olav.Nygaard@uia.no}
\urladdr{http://home.uia.no/olavn/}
\thanks{The research of J. Langemets was supported by Estonian Science Foundation Grant 8976,  Estonian Targeted Financing Project SF0180039s08 and Estonian Institutional Research Project IUT20-57.}
%
\begin{abstract}
 The aim of this note is to complement and extend some recent results on
Whitley's indices of thinness and thickness in three main directions. Firstly, we investigate both the indices when forming $\ell_p$-sums of
Banach spaces, and obtain formulas which show that they behave rather
differently. Secondly, we consider the relation of the indices of the space
and a subspace.  Finally, every Banach space $X$ containing
a copy of $c_0$ can be equivalently renormed so that in the new norm $c_0$
is an M-ideal in $X$ and both the thickness and thinness index of $X$ equal 1.

\end{abstract}
\maketitle

\section{Introduction}

Let $X$ be a Banach space, $B_X$ its unit ball and $S_X$ its unit sphere. Also, denote by $B(x,r)$ the closed ball with center in $x$ and radius $r$. Whitley introduced in \cite{W} the \emph{index of thickness},
\[T_W(X) = \inf\left\{r >0: \exists (x_i)_{i=1}^n \subset S_X \:\mbox{with}\: S_X \subset \bigcup_{i=1}^n B(x_i, r)\right\},\]
and the \emph{index of thinness},
\[t(X) = \inf\left\{r > 0: \forall (x_i)_{i=1}^n \subset S_X, \eps>0,\exists x \in S_X \:\mbox{with}\: \max_i \| x_i -  x\| < r + \eps \right\}.\]

The subscript $W$ in $T_W(X)$ is to indicate that this is Whitley's original definition. As is easily observed, if $\dim X<\infty$, $T_W(X)=0$ and $t(X)=2$ while if $\dim X=\infty$, $T_W(X),t(X)\in[1,2]$. More difficult is the fact, proved by Whitley, that
\[T_W(\ell_p)=2^{1/p}=t(\ell_p), \hspace{1cm} 1\leq p<\infty.\] 
Together with Whitley's observations that $T_W(c_0)=1=T_W(\ell_\infty),t(c_0)=1$ and $t(\ell_\infty)=2$ it is clear that the whole range $[1,2]$ of indices is possible and that $(1,2)$ is covered by indices of reflexive spaces. We will see that, by choosing appropriate reflexive spaces $X$ and $Y$, we may have $T_W(X)=1$ and $t(Y)=2$, but never $T_W(X)=2$ nor $t(X)=1$.

Whitley \cite[Lemmas~3 and 8]{W} also showed that $t(L_\infty[0,1])=T_W(L_\infty[0,1])\\=2$.
Recently (see \cite{BJ} and \cite{CPS}) it was shown that $T_W(L_p[0,1])=2^{1/p}$ for
$1 \le p < \infty$. In \cite[Example~3.6]{BSP} it was shown that $t(L_1[0,1])=2$ and
in \cite[Theorem~6.3]{MP} that $t(L_p[0,1]) = 2^{1/p}$ for $p \ge 2$.
In fact, it is clear from the proof of \cite[Theorem~6.3]{MP} that
$t(L_p[0,1]) \le 2^{1/p}$ for all $1 \le p < \infty$. Rainis Haller (private communication) pointed out to us that for all $f \in S_{L_p}$ $\|f_i - f\|^p$ is almost $2$ when $f_i =
n^{1/p}\chi_{[i/n,(i+1)/n]}$.
This shows that the lower bound is also $2^{1/p}$, hence $t(L_p[0,1])=2^{1/p}$ for
all $1 \le p < \infty$.

Before proceeding, let us just mention that in \cite{CPS} it is noted that when $\dim X=\infty$ and $S_X\subset\bigcup_{i=1}^n B(x_i,r), (x_i)_{i=1}^n\subset S_X$, then $B_X\subset\bigcup_{i=1}^n B(x_i,r)$. Thus, for $\dim X=\infty$, the index
\[T(X) = \inf\left\{r >0: \exists (x_i)_{i=1}^n \subset S_X \:\mbox{with}\: B_X \subset \bigcup_{i=1}^n B(x_i, r)\right\}\]
equals $T_W(X)$. Note that when $\dim X<\infty$ we always have $T(X)=1$ (while $T_W(X)=0)$. In this note we are only interested in calculating the index for infinite-dimensional Banach spaces and will thus take the freedom to use $T(X)$ in what follows to denote also $T_W(X)$.

Most of what is known concerning $T(X)$ and $t(X)$ can be found by combining \cite{W}, \cite{BJ} and \cite{CPS} (note that the two latter overlap a bit on $T$-results). The particular case when $X$ is separable and $T(X)=2$ is thoroughly described in terms of the almost Daugavet property in \cite{KSW} and \cite{Lu}. For the non-separable case see  \cite{HL}. 

Yost introduced in \cite{Y} two indices
\[
\mu_1(X) = \sup_{\overset{x_1,\ldots,x_n \in S_X}{n \in \mathbb{N}}}
\inf_{x \in S_X} \frac{1}{n} \sum_{i=1}^n \|x_i - x\|
\]
and
\[
\mu_2(X) = \inf_{\overset{x_1,\ldots,x_n \in S_X}{n \in \mathbb{N}}}
\sup_{x \in S_X} \frac{1}{n} \sum_{i=1}^n \|x_i - x\|.
\]
He showed that we always have  $\mu_1(X)\leq \mu_2(X)$ for any Banach space $X$.

Note that we can give similar formulations to the thinness and thickness index 
\[
t(X) = \sup_{\overset{x_1,\ldots,x_n \in S_X}{n \in \mathbb{N}}}
\inf_{x \in S_X} \max_{1 \le i \le n} \|x_i - x \|
\]
and
\[
T(X) = \inf_{\overset{x_1,\ldots,x_n \in S_X}{n \in \mathbb{N}}}
\sup_{x \in S_X} \min_{1 \le i \le n} \|x_i - x \|.
\]

From these definitions we observe that $1\leq\mu_1(X)\leq t(X)\leq 2$ and $1\leq T(X)\leq \mu_2(X)\leq 2$ for any Banach space $X$. 

The following example by Papini (see \cite[Example 2]{P}) shows that we may have $\mu_1(X)< t(X)$ and $T(X)< \mu_2(X)$.

\begin{eks}
Let $K=\{c\}\cup [a,b]$ with $c\notin[a,b]$ and consider $X=C(K)$. Then we have that $\mu_1(X)=\mu_2(X)=3/2$ while $T(X)=1$ and $t(X)=2$.
\end{eks}

It turns out that $T(X)=2$ is equivalent to $\mu_2(X)=2$ (see \cite[Theorem~2.1]{P}). It is clear that if $t(X)=1$ then  $\mu_1(X)=1$, but the reverse implication seems to be unknown.
\begin{ques}
Does $\mu_1(X)=1$ imply $t(X)=1$?
\end{ques}

Let us now present our contribution to the present theory of indices of thickness and thinness.

We start with some preparatory observations that we will use
throughout. According to \cite{ALL}, see Proposition~3.3, a Banach space $X$ is \emph{almost square} if for every finite subset $(x_i)_{i=1}^n\subset S_X$ and $\eps > 0$ there exists $y \in S_X$ such that $\|x_i \pm y\| \le 1 + \eps$ for $i=1,2,\ldots,n$. From \cite{HLP}, see Proposition~2.4, $X$ is  \emph{octahedral} if for every finite subset $(x_i)_{i=1}^n\subset S_X$ and every $\eps>0$ there exists $y\in S_X$ such that $\|x_i\pm y\|>2-\eps$ for every $1,2,\ldots, n$. The definition of octahedrality goes back to \cite[p.~12]{G}; we have not used the original definitions in order to have comparable formulations of almost squareness and octahedrality. Two basic observations are that $t(X)=1$ is equivalent to $X$ being almost square and that $T(X)=2$ is the same as $X$ being octahedral. We also prove two lemmas on thickness and thinness of $c_0$-sums of Banach spaces and continue with a result that can be informally expressed as ``thin spaces can have whatever thickness you like''. More precisely, we prove that for every $\alpha\in[1,2]$, there is a Banach space $X$ with $T(X)=\alpha$ while $t(X)=1$. Then we will see that there are infinite-dimensional reflexive spaces $X$ and $Y$ for which $T(X)=1$ and $t(Y)=2$. We end Section 2 by studying the behaviour of thinness index when forming $\ell_p$-sums of Banach spaces (see Propositions~\ref{t-psum} and \ref{t-infsum}).

Next we address the problem of the relation between the thickness and thinness indices of the space and a subspace. Our question is motivated by the easy observation that $T(X)\geq T(X^{\ast\ast})$ for any Banach space $X$. To generalize this observation we will need some concepts. First, recall the definition of an ideal (or locally 1-complemented subspace):

\begin{defn}\label{def:super1comp} Let $X$ be a Banach space and $Y$ a subspace. $Y$ is called an \em ideal \em in $X$ if  for every $\eps>0$ and every finite-dimensional subspace $E\subset X$ there exists $T:E\to Y$ such that
\begin{itemize}
  \item [(i)] $Te=e$ for all $e\in Y\cap E$.
  \item [(ii)] $\|Te\|\leq (1+\eps)\|e\|$ for all $e\in E$.
\end{itemize} 
\end{defn}  
\noindent If we instead of (ii) above have the stronger condition
\begin{itemize}
  \item [(ii')] $(1+\eps)^{-1}\|e\|\leq\|Te\|\leq (1+\eps)\|e\|$ for all $e\in E$,
\end{itemize} 
then the ideal is called an \emph{almost isometric ideal} (ai-ideal). So, in other words, an ai-ideal is an ideal (a locally 1-complemented
subspace) where the local projections can be taken as almost
isometries. The notion of an ai-ideal was defined and studied in \cite{ALN}. 

With every ideal $Y\subset X$ there is an associated ideal
projection $P:\Xast\to\Xast$ with $\|P\|=1$ and $\ker P=Y^\perp$. It
is observed in \cite[Proposition~2.1]{ALN} that the local projections
can be taken as almost isometries whenever $P\Xast$ is a 1-norming subspace of $\Xast$, that is, when the ideal is \emph{strict}. There are, however, ai-ideals which are not strict, see \cite[Example~1]{ALN} or Remark~\ref{rem:gurarii} below. 

Any Banach space $X$ is an ai-ideal in $\Xastast$; this fact is usually referred to as the principle of local reflexivity. Our generalization of the observation $T(X)\geq T(X^{\ast\ast})$ is that ai-ideals in $X$ always have at least as high thickness as the space itself. Also these subspaces always have lower or equal thinness index. 

It is well-known that for a Banach space $X$ we have $\Xast = L_1(\mu)$ for some measure $\mu$, i.e. $X$ is \emph{Lindenstrauss}, if and only if $X$ is an ideal in every superspace. Some spaces are even ai-ideals in every superspace; these spaces are the Gurari{\u\i}-spaces (see \cite[Theorem~4.3]{ALN}). Being an ai-ideal in every superspace will imply that any Gurari{\u\i}-space has thickness index 2 and thinness index 1. 

For our last and perhaps most interesting result, recall the definition of an M-ideal: A subspace $Y$ of $X$ is called an \emph{M-ideal} if the associated ideal projection $P:\Xast\to\Xast$ is an L-projection, that is,
\[\|\xast\|=\|P\xast\|+\|(I-P)\xast\|\:\:\mbox{for all}\:\:\xast\in\Xast.\]

Proposition II.2.10 in \cite{HWW} tells us: If $Y\subset X$ is a subspace isometric to $c_0$, the original norm on $Y$ can be extended to an equivalent norm on $X$ in such a way that $Y$ becomes an M-ideal in $X$ with the new norm. 

We will prove that any Banach space $X$ containing an isomorphic copy of $c_0$ can be equivalently renormed so that, in this new norm, $c_0$ becomes an M-ideal in $X$ and, moreover, both the thickness and thinness index of $X$ equal 1.

\begin{rem}\label{rem:T2-renorm}
That $X$ can be equivalently renormed to have $T(X)=2$ if and only if
$X$ contains an isomorphic copy of $\ell_1$ was proved in \cite[Theorem~9.2]{GoKa}. See also \cite{HL} or \cite{P}.
\end{rem}

Sometimes we will say that $X$ is thick, meaning $T(X)=2$. Correspondingly, we sometimes say thin, meaning $t(X)=1$. But, as can be seen e.g. from Proposition~\ref{prop1} or Proposition~\ref{prop:gurarii} below, Banach spaces may very well be both thick and thin at the same time, so the terms should not be read too literally.

\section{Some preparatory observations and thin spaces with any kind of thickness}

One can easily observe that a Banach space $X$ satisfies the condition $t(X)=1$ if and only if it is almost square (see \cite[Proposition~3.3]{ALL}). Almost square spaces were introduced and studied in \cite{ALL}. In \cite{ALL} it is proved that if $X$ is almost square, then every convex combination of slices of $B_X$ has diameter 2. It is more or less folklore (see \cite[p.~12]{G}) that this slice property is in turn equivalent to $\Xast$ being octahedral. In \cite{GoKa} it is proved that octahedrality is equivalent to the condition that whenever the unit ball is covered by a finite number of balls, one of those balls already contains the unit ball itself, hence the thickness is 2. Thus we have 

\begin{prop}\label{prop:t-T} If $t(X)=1$, then $T(X^\ast)=2$.
\end{prop} 

\begin{rem}
The converse implication, $T(\Xast)=2\Rightarrow t(X)=1$, is not true. As an example, take $X=C[0,1]$. Then $T(\Xast)=2$ since $\Xast$ is octahedral, and also $t(X)=2$ by \cite[Lemma~8]{W}.
\end{rem}

We now study thinness and thickness indices of $c_0$-sums of sequences of Banach spaces. Let us first observe that $c_0$-sums are always thin.

\begin{lem}\label{lemma2} If $(X_n)$ is a sequence of Banach spaces, then $t(c_0(X_n))=1$.
\end{lem}

\begin{proof}   Let $(x_i)_{i=1}^n \in S_{c_0(X_n)}$ and $\varepsilon > 0$.
  Find $N$ such that $\|x_i(n)\| < \varepsilon$
  for all $n \ge N$. Choose any $y_N \in S_{X_N}$
  and define $y = (0,\ldots,y_N,0,\ldots) \in S_{c_0(X_n)}$.
  Then $\max_i \|x_i - y\| < 1 + \varepsilon$, and so the lemma is proved.
\end{proof}

\begin{rem} Note that there appears to be a misprint in \cite[Lemma 4.1]{BJ}. The authors state that $t(\ell_\infty(X_n))=1$, but with $X_n=\mathbb{R}$, we have $t(\ell_\infty)=2$ (see \cite[Lemma 8]{W}). 
 \end{rem}

\begin{rem}
The thinness of a subspace may be strictly bigger than the thinness of the space itself. Indeed, $t(\ell_1)=2$, but $t(c_0(\ell_1))=1$ by Lemma~\ref{lemma2}.
\end{rem}

The corresponding result on the thickness index is harder. The proof is essentially that of \cite[Theorem~2 (3)]{CPS} and is omitted. 

\begin{lem}\label{lemma1}
  If $(X_n)$ is a sequence of Banach spaces then
  $T(c_0(X_n)) = \inf_n T(X_n)$.
\end{lem}

\begin{rem} Observe that Lemma~\ref{lemma1} implies that there is in fact equality in \cite[Proposition~2.14 (1)]{BJ}. 
\end{rem}

With the observations we have made so far at hand, we get the following result:

\begin{prop}\label{prop1} For every $\alpha\in[1,2]$ there is a Banach space $X$ with $T(X)=\alpha$ while $t(X)=1$ and $T(\Xast)=2$.
\end{prop}

\begin{proof} The statement ``and $T(\Xast)=2$'' is Proposition~\ref{prop:t-T}. From Whitley's paper (\cite[Lemma~4]{W}) we know that
  $T(\ell_p) = 2^{1/p}$ for $1 \le p < \infty$.
  From Lemma~\ref{lemma1} we get that also $T(c_0(\ell_p))=2^{1/p}$. 	
  From Lemma~\ref{lemma2} we know that $t(c_0(\ell_p))=1$. Thus the result has been proved for all $\alpha\in(1,2]$. For $\alpha=1$ consider $X=c_0$.
\end{proof}

It is clear that $t(X)>1$ and $T(X)<2$ for all reflexive Banach spaces, this follows from e.g. Remark \ref{rem:T2-renorm} and Proposition \ref{prop:t-T} above.
The next proposition shows that all other possible values of $t(X)$ and $T(X)$ are covered by infinite-dimensional reflexive spaces.

\begin{prop}\label{proprefl}
  For every $\alpha \in [1,2)$ there is an infinite-dimensional reflexive Banach space $X$
  with $T(X)=\alpha$, and
  for every $\alpha \in (1,2]$ there is an infinite-dimensional reflexive Banach space $X$
  with $t(X)=\alpha$.
 \end{prop}

\begin{proof}
  As we noted in the introduction,
  Whitley showed that $T(\ell_p)=2^{1/p}=t(\ell_p)$ for $1< p<\infty$
  and this covers the interval $(1,2)$.

  Let $Y$ be any infinite-dimensional reflexive Banach space.
  If we let $X=Y\oplus_\infty\mathbb{R}$, then it
  follows easily from (the proof of) \cite[Lemma~3]{CPS}
  that $T(X)=1$.
  On the other hand if we let $X=Y\oplus_1\mathbb{R}$
  then $t(X) = 2$ by Corollary~\ref{t-1sum} below, since $t(\mathbb{R})=2$.
\end{proof}

For $\ell_p$-sums we have the following result:
\begin{prop}\label{t-psum}
  Let $Y$ and $Z$ be Banach spaces and let $1 \le p < \infty$.
  Then $X = Y \oplus_p Z$ satisfies $t(X)\geq((t(Y)-1)^p+1)^{1/p}$.
\end{prop}

\begin{proof}
In \cite[Lemma~5.6 and the preceding remark]{ALL} it is noted that for $1\leq p<\infty$, $X\oplus_p Y$ is never almost square, i.e. $t(X\oplus_p Y)>1$. Since $((t(Y)-1)^p+1)^{1/p}=1$ when $t(Y)=1$ we may assume that $t(Y)>1$.

For any $1<\alpha < t(Y)$
  there exist $(y_i)_{i=1}^n \subset S_Y$
  and $\varepsilon > 0$ such that
  $\max_i \|y-y_i\| \ge \alpha + \varepsilon$
  for all $y \in S_Y$.
  For a given $\alpha < t(Y)$ consider $(y_i)_{i=1}^n \subset S_Y$
  and $\varepsilon > 0$ as above and
  define $x_i = (y_i,0) \in S_X$ for $i=1,2,\ldots,n$.

  For $x = (y,z) \in S_X$ we have
  $\|x\|^p = \|y\|^p + \|z\|^p = 1$. We will also need that
  \begin{equation*}
    1 - \|y\|
    = \frac{(1-\|y\|)}{\|y\|}\|y\|
    = \| \frac{y}{\|y\|} - y\|.
  \end{equation*}
  By the triangle inequality and monotonicity of the $\ell_p$-norms
  \begin{align*}
    \max_i\|x_i - x\|^p &= \max_i \|y_i-y\|^p + \|z\|^p \\
    &\ge \bigl(\max_i \|y_i-\frac{y}{\|y\|}\| -
    \| \frac{y}{\|y\|} - y\| \bigr)^p + \|z\|^p \\
    &\ge (\alpha + \varepsilon - 1 + \|y\| )^p + \|z\|^p \\
    &\ge (\alpha+\varepsilon-1)^p + \|y\|^p + \|z\|^p
    = (\alpha+\varepsilon-1)^p + 1.
  \end{align*}
  Since $\alpha < t(Y)$ and $x \in S_X$ are arbitrary, we obtain
  $t(X) \ge ((t(Y)-1)^p+1)^{1/p}$.
\end{proof}

\begin{rem}  As a general lower bound this is best possible since $t(\ell_p \oplus_p X)= 2^{1/p}$ by \cite[Proposition~4.3]{BJ} for any space with $t(X)=2$, for example  $X=\ell_1$.
\end{rem}

\begin{cor}\label{t-1sum}
\mbox{} 
\begin{itemize}
\item [(i)] If $X$ and $Y$ are Banach spaces, then $t(X \oplus_1 Y) \ge \max\{t(X),t(Y)\}$.
\item [(ii)] If $(X_j)^{\infty}_{j=1}$ is a sequence of non-trivial Banach spaces, then\\ $t(\ell_p(X_j))\geq\sup_j((t(X_j)-1)^p+1)^{1/p}$. Moreover, if $\sup_j t(X_j)=2$, then $t(\ell_p(X_j))=2^{1/p}$.
\end{itemize}
\end{cor}

\begin{proof}
It is clear that (i) holds. For the moreover part in (ii) it suffices to observe that the upper bound is proved in \cite[Lemma~4.1]{BJ}.
\end{proof}

\begin{prop}\label{t-infsum}
  Let $X$ and $Y$ be a Banach spaces.
  Then $t(X \oplus_\infty Y) = \min\{t(X),t(Y)\}$.
\end{prop}

\begin{proof}
Let $\alpha$ and $\beta$ be such that $\alpha<t(X)$ and $\beta<t(Y)$. Then there exist $(x_i)_{i=1}^n \subset S_X$, $(y_j)_{j=1}^k \subset S_Y$ and $\eps>0$ such that $\max_i \|x_i - x\| \ge \alpha+\eps$ for all
  $x \in S_X$ and  $\max_j \|y_j - y\| \ge \beta+\eps$ for all
    $y \in S_Y$.

  Without loss of generality we may assume that $k=n$ by just repeating
  some vectors.
  Define $z_i = (x_i,y_i)$, $1 \le i \le n$.
  Let $z = (x,y) \in X \oplus_\infty Y$
  with $\|z\| = 1$. Then either $\|x\|=1$
  or $\|y\|=1$ and hence
  \begin{equation*}
    \max_i \|z_i - z\|
    = \max_i\{\|x_i-x\|,\|y_i-y\|\}
    \ge \min\{\alpha,\beta\}+\eps.
  \end{equation*}
  Thus $t(X \oplus_\infty Y) \ge \min\{t(X),t(Y)\}$.
  
 Note that the following holds in every Banach space: If two elements
 $x'$ and $x$ have norm one and $\|x' - x\| < a$ where $a\ge 1$, then
 for all $0 \le r \le 1$ we have $\|rx' - x\| < a$. Indeed, 
 \[
 \|rx' -  x\| =  \|rx' - rx + rx - x\| \le r\|x' - x\| + 1-r < a.
 \]
 
 Now, suppose $\min(t(X), t(Y)) = t(X)$ and let $\eps > 0$. Let $(x_i,
 y_i)_{i=1}^n$ be a finite set in the unit sphere of $X \oplus_\infty Y$. Let $u_i=x_i/\|x_i\|$ if $x_i\neq 0$. Then there is an element $x\in S_X$ such that $\max_{i}\|u_i-x\|<t(X)+\eps$. Consider the element $(x,0)$ from the unit sphere of $X \oplus_\infty Y$. By the previous
 paragraph we get 
 \[
 \max_i\|(x_i,y_i) - (x, 0)\| = \max_i\{\|\|x_i\|u_i-x\|,\|y_i\|\} < t(X) + \eps.
 \]
 Finally, if $x_i=0$ for every $i$, then for any $x\in S_X$ we have 
 \[
 \|(0,y_i)-(x,0)\|=1\le t(X).
 \]
\end{proof}

The aim of the following example is to show that although $T(\ell_1)=t(\ell_1)=2$, then by forming $\ell_p$-sums these indices behave quite differently. In \cite[Lemma~2]{CPS} it was shown that $T(\ell_1\oplus_2 \ell_1)=\sqrt{2+\sqrt{2}}$, but $t(\ell_1\oplus_2 \ell_1)=2$ as we will see from the next proposition.

\begin{prop}
$t(\ell_1\oplus_p \ell_1)=2$, $1\leq p\leq \infty$.
\end{prop}
\begin{proof}
If $p=1$ or $p=\infty$ then the statement follows from Corollary~\ref{t-1sum} or Proposition \ref{t-infsum}, respectively, since $t(\ell_1)=2$. Suppose now that $1<p<\infty$.

  Let $a^p + b^p = 1$. Then
  \begin{equation*}
    \|(\pm a e_1, \pm b e_1)\|^p
    = a^p + b^p = 1.
  \end{equation*}
  We can parametrize by $a = (\cos \theta)^{2/p}$
  and $b = (\sin \theta)^{2/p}$.
  For any $x \in \ell_1$ and $a \in [0,1]$ we have
  \begin{equation*}
    \max_{\pm} \|\pm a e_1 - x\|
    = \max_{\pm} |a-x_1| + \sum_{n=2}^\infty |x_n|
    = a + |x_1| + \sum_{n=2}^\infty |x_n|
    = a + \|x\|.
  \end{equation*}
  Hence for $z = (x,y) \in S_Z$ we have
  $\|y\|^p = 1 - \|x\|^p$ and
  \begin{multline}\label{eq:3}
    \max_{\pm} \|(\pm a e_1, \pm b e_1)-(x,y)\|^p
    =
    \left( a + \|x\| \right)^p
    +
    \left( b + (1-\|x\|^p)^{1/p} \right)^p
    \\
    =\left( (\cos \theta)^{2/p} + \|x\| \right)^p
    +
    \left( (\sin \theta)^{2/p} + (1-\|x\|^p)^{1/p} \right)^p.
  \end{multline}
  Consider the continuous function on $[0,\pi/2] \times [0,1]$
  defined by
  \begin{equation*}
    f(\theta,\xi) = 
    \left( (\cos \theta)^{2/p} + \xi \right)^p
    +
    \left( (\sin \theta)^{2/p} + (1-\xi^p)^{1/p} \right)^p.
  \end{equation*}
  For $\theta = \arccos(\xi^{p/2})$ we have
  \begin{equation*}
    f(\arccos(\xi^{p/2}),\xi)
    = (2\xi)^p +
    (2(1-\xi^p)^{1/p})^p
    = 2^p.
  \end{equation*}
  Let $\varepsilon > 0$. Since $f(\theta,\xi)$ is uniformly
  continuous there exists $\delta > 0$ such that
  \begin{equation*}
    \|(\theta,\xi)-(t,u)\| < \delta
    \Rightarrow
    |f(\theta,\xi)-f(t,u)| < \varepsilon.
  \end{equation*}
  Choose a $\delta$-net $(\xi_i)_{i=1}^n$ for the interval $[0,1]$
  and define $\theta_i = \arccos(\xi_i^{p/2})$.
  
  Let $a_i = (\cos \theta_i)^{2/p}$ and $b_i = (\sin \theta_i)^{2/p}$ then
  for $1 \le i \le 4n$ and $1\le k\le n$ define
  $z_{i} = (\pm a_k e_1,\pm b_k e_1)$.

  For any $z = (x,y) \in S_Z$ we have by \eqref{eq:3}
  \begin{equation*}
    \max_i \|z_i - z\|^p = \max_i f(\theta_i,\|x\|).
  \end{equation*}
  Choose $\xi_j$ such that $|\xi_j - \|x\|| < \delta$.
  Then
  \begin{equation*}
    \max_i \|z_i - z\|^p = \max_i f(\theta_i,\|x\|)
    \ge f(\theta_j,\|x\|)
    \ge f(\theta_j,\xi_j) - \varepsilon
    = 2^p - \varepsilon.
  \end{equation*}
  This shows that $t(\ell_1 \oplus_p \ell_1) \ge
  \sqrt[p]{2^p-\varepsilon}$.
  Since $\varepsilon > 0$ was arbitrary we must have
  $t(\ell_1 \oplus_p \ell_1) = 2$.
\end{proof}

\section{Thinness and thickness of almost isometric ideals}

Using Goldstine's theorem it is easy to see that $T(X^{\ast\ast})\leq T(X)$. This inequality may be strict. As an example $T(C[0,1])=2$ while $T(C[0,1]^{\ast\ast})=1$ by \cite[Lemma~3]{W} since $C[0,1]^{\ast\ast}$ is not octahedral and can be viewed as a $C(K)$ space (cf. e.g. \cite[Theorems~4.3.7 and 4.3.8]{AK}). Note that this example answers a question in \cite{CP} whether we always have $T(X)=T(\Xastast)$. We will now put these observations into a broader perspective. A Banach space $X$ is always an ai-ideal in $X^{\ast\ast}$ and so the observation above that $T(X^{\ast\ast})\leq T(X)$ is a very
particular case of the following proposition.

\begin{prop}\label{prop2}
  If $Y$ is an ai-ideal in $X$, then $T(X) \le T(Y)$ and $t(Y)\leq t(X)$.
\end{prop}

\begin{proof}
  Assume that $(y_i)_{i=1}^n$ is an $r$-net for $S_Y$.
  Let $\varepsilon > 0$.
  Let $x \in S_X$ and $E = \linspan((y_i),x)$.
  Find an $\varepsilon$-isometry $T : E \to Y$.
  Let $z = Tx/\|Tx\| \in S_Y$. Then $\|z - Tx\| \le \varepsilon$
  since $(1+\varepsilon)^{-1} \le \|Tx\| \le 1 + \varepsilon$.
  Now find $j$ such that $\|y_j - z \| \le r$, then
  \begin{equation*}
    \|y_j - x\| \le (1+\varepsilon)\|y_j - Tx\|
    \le (1+\varepsilon)(r + \varepsilon).
  \end{equation*}
  Since $T(X)$ is an infimum and $\varepsilon > 0$ is arbitrary,
  we get $T(X) \le T(Y)$.
  
  Next assume that $(y_i)_{i=1}^n \subset S_Y$.
  Let $\varepsilon > 0$.
  Find $x \in S_X$ such that
  $\max \|y_i - x\| < t(X) + \varepsilon$.
  Let $x \in S_X$ and $E = \linspan((y_i),x)$.
  Find $\varepsilon$-isometry $T : E \to Y$.
  Let $z = Tx/\|Tx\| \in S_Y$. Then, as above, $\|z - Tx\| \le \varepsilon$. Now
  \begin{equation*}
    \|y_j - z\| \le
    \|y_j - Tx\| + \varepsilon \le
    (1+\varepsilon)\|y_j - x\| + \varepsilon
    \le (1+\varepsilon)(t(X) + \varepsilon).
  \end{equation*}
Since $t(X)$ is an infimum and $\varepsilon > 0$ is arbitrary
  we have shown that $t(Y) \le t(X)$.
\end{proof}

An ai-ideal may well have strictly less thinness than its super-space.
\begin{rem}
 Note that $t(c_0)=1$ while $t(\ell_\infty)=2$, so we have that $t(X)<t(\Xastast)$ for $X=c_0$.
\end{rem}

Proposition \ref{prop2} will turn out to provide us with a class of spaces which are both thick and thin at the same time, namely the Gurari{\u\i}-spaces. The ``up to date reference'' for Gurari{\u\i}-spaces is \cite{GK} and the definition of a Gurari{\u\i}-space can be found there. We will, however, use the alternative description of Gurari{\u\i}-spaces (\cite[Theorem~4.3]{ALN}): The Gurari{\u\i}-spaces is exactly the class of Banach spaces with the property that they form an ai-ideal in every super-space. It is known that there is only one separable Gurari{\u\i}-space (up to linear isometry), and that a Gurari{\u\i}-space $X$ is universal in the sense that it contains an isometric copy of all Banach spaces $Y$ with $\dens(Y)\leq \dens(X)$ (see \cite{GK}). Gurari{\u\i} constructed the first separable such space in 1966 in his seminal paper \cite{Gu}.

\begin{prop}\label{prop:gurarii} If $X$ is a Gurari{\u\i}-space, then
  $T(X)=2$ and $t(X)=1$.
\end{prop}

\begin{proof}  Let $X$ be a Gurari{\u\i}-space.
  Then, by \cite[Theorem~4.3]{ALN}, $X$ is an ai-ideal in any super-space.
  In particular, it is an ai-ideal in $Y=C(B_{X^\ast},w^\ast)$
  (which does not have isolated points). Thus, by \cite[Lemma~3]{W}
  and Proposition~\ref{prop2} above, $T(X)\geq T(Y)=2$.

  To see that $t(X)=1$, just note that, by \cite[Theorem~4.3]{ALN},
  $X$ is an ai-ideal in $c_0(X)$ ($X$ is isometrically isomorphic to
  $(X,0,0,\ldots)$), and the result follows from Lemma~\ref{lemma2}
  and Proposition~\ref{prop2}.
\end{proof}

\begin{rem}\label{rem:gurarii} Note that the embedding $X\to c_0(X)$ in the proof of
  Proposition~\ref{prop:gurarii} also makes $X$ a 1-complemented
  subspace of $c_0(X)$. Thus, Gurari{\u\i}-spaces $X$, embedded this
  way in $c_0(X)$, provide examples of non-strict, ai-ideals (see also \cite[Example~1]{ALL} where a non-strict ai-ideal in $c_0$ is constructed).   
\end{rem}

We have already mentioned that Gurari{\u\i}-spaces are Lindenstrauss
spaces. Lindenstrauss proves in his famous memoir (see \cite[Theorem~6.1]{Li}) that when $X^\ast=L_1(\mu)$ and $B_X$ has extreme points, then $X = X_1$, where $X_1$ is a subspace of some $C(K)$ space that contains $1$.
Using the unit, we see that \emph{$t(X)=2$ whenever $X$ is an $L_1$-predual with $\ext(B_X)\neq\emptyset$}. In particular we get:

\begin{prop} If $X$ is a Gurari{\u\i}-space, then $\ext(B_X)=\emptyset$.
\end{prop}

We have just argued that a Lindenstrauss space has thinness index 2 as soon as $\ext(B_X)\neq\emptyset$. 

\begin{prop} For every $\alpha \in [1,2]$, there is a Lindenstrauss space with $t(X) = \alpha$. For any Lindenstrauss space we have $T(X^\ast) = 2$ and $t(\Xastast)=2$. 
\end{prop}

\begin{proof} $T(\Xast)=2$ because $\Xast$ is octahedral and $t(\Xastast)=2$ because $\Xastast$ is a Lindenstrauss space with $\ext(B_X)\neq\emptyset$. We need to consider the case $\alpha\in(1,2)$. For this, let $r>1$ and $X_r=\{f \in C[0,1] : f(0) = rf(1)\}$. Then the spaces $X_r$ are all $L_1$-preduals (see e.g., \cite[p.~83]{HWW}). We are going to show that $t(X_r)=1+\frac{1}{r}$. Note that for all $f\in B_{X_r}$ we have $|f(1)|\leq\frac{1}{r}$.

To see that $t(X_r)\geq 1+\frac{1}{r}$ let $f_1(x) = (1-x) + \frac{1}{r}x$ and $f_2(x) = -f_1(x)$. If $\|g\|=1$, then there is a point $x_0$ where $|g(x_0)| =1$. Without loss of generality assume $g(x_0) = -1$. Then $\frac{1}{r} + 1 \leq |f_1(x_0) - g(x_0)|  \leq \max \|f_i - g\|$. Hence $t(X_r) \geq 1+\frac{1}{r}$.

To see that $t(X_r)\leq 1+\frac{1}{r}$ let $f_1,f_2,\ldots,f_n \in S_{X_r}$ and $\eps > 0.$ Find an interval $(a,1)$ where $|f_i(x)| < \frac{1}{r} + \eps.$ Now choose any $g \in S_X$ with support on $(a,1)$. Then $\|f_i - g\| < 1 + \frac{1}{r} + \eps$, hence $t(X) \leq 1+ \frac{1}{r}.$
\end{proof}

\section{M-ideal-renorming of copies of $c_0$}

Recall that $Y$ is an M-ideal in $X$ if the ideal projection $P : X^* \to X^*$, with $\|P\|=1$ and $\ker P = Y^{\perp}$, is an L-projection. If $X$ is an M-ideal in $\Xastast$ (like $c_0$ is in $\ell_\infty$),
$X$ is called an \emph{M-embedded} space.

M-ideals play a very important role in Banach space theory; the main
reference for the theory of M-ideals is \cite{HWW}. When $X$ is a
Banach space and $X$ contains an isometric copy of $c_0$, the
$c_0$-norm can always be extended to all of $X$ such that, in
this new norm, $c_0$ is an M-ideal in $X$, see
\cite[Proposition~II.2.10]{HWW}. We will now prove that we can extend in
such a way that $T(X)=t(X)=1$ in the new norm. The proof is based on an
idea which appears in \cite[Theorem~3.14]{ALL}, which in turn used ideas from \cite[Lemma 2.3]{BLZ}.

\begin{thm}\label{thm:c0-renorm-Mid}
  If $X$ contains an isomorphic copy of $c_0$, then $X$ can be renormed
  so that, in this new norm, $c_0$ becomes an M-ideal in $X$ and $T(X)=t(X)=1$.
\end{thm}

\begin{proof}
  First, \cite[Lemma II.8.1]{DGZ}, we can
  renorm $X$ so that it contains an isometric copy of $c_0$. Denote by
  $\|\cdot\|$ this new norm on $X$. Let 
  \[A = \{Y \subset X: c_0 \subset Y, Y \text{ separable}\},\]
  and order $A$ by inclusion, i.e., $Y_2 \le Y_1$ if $Y_2 \subset Y_1$. 
  For every $Y \in A$ 
  there exists, by Sobczyk's theorem, a projection $P_Y$
  onto $c_0$ with norm $2$ or less. Let $P_Y$ be such a projection and
  for each $Y \in A$ and $x \in Y$ let
  \[\|x\|_Y:= \max\{\|P_Y(x)\|, \|x - P_Y(x)\|\}.\]
  By letting $||x||_{Y} = 0$ for $x \notin Y$ we can consider $(||x||_Y)_{Y\in A}$
  as a net in $\Pi_{x \in X} [0,3||x||]$. By Tychonoff's theorem this
  net has a convergent subnet, still denoted $(||x||_Y)_{Y\in A}$, and we
  may define \[ |||x||| = \lim_Y ||x||_Y. \]
  
  It is straightforward to show that $\tn \cdot\tn $ is a norm on $X$
  which satisfies $\frac{1}{2}\|x\| \le \tn x\tn  \le 3\|x\|$. Also
  $\tn \cdot\tn$ extends the max norm $\|\cdot\|$ on $c_0$. It was shown in \cite[Theorem~3.14]{ALL} that this norm is almost square, i.e., $t(X)=1$ in this norm.

  We want to show that $c_0$ is an M-ideal in $X$ in this new norm.
  Let $x \in B_{(X,\tn \cdot\tn)}$, $y_1,y_2,y_3 \in B_{c_0}$ and $\varepsilon > 0$.
  Let $y_0 = 0$.

  Let $(z_n)$ be a sequence which is dense in $c_0$ and let $z_0 = 0$.
  Let $(\varepsilon_n)_{n=1}^\infty$ be a strictly decreasing null
  sequence of positive reals.

  Let $Y_0 = \linspan\{x, c_0\}$ and choose $Y_1 \in A$
  with $Y_1 \supset Y_0$ such that
  \begin{equation*}
    \left| \; \tn x + y_i - z_0 \tn  - \|x + y_i - z_0\|_{Y_1} \right| < \varepsilon_1
  \end{equation*}
  for $i=0,1,2,3$.
  Then for $n \ge 1$ inductively choose $Y_{n+1}  \in A$ with $Y_{n+1} \supset
  Y_n$ such that
  \begin{equation*}
    \left| \; \tn x + y_i - z_k \tn  - \|x + y_i - z_k\|_{Y_n} \right| < \varepsilon_n
  \end{equation*}
  for every $k \le n$ and $i=0,1,2,3$.
  (Note that the inequality above holds also for every $Y \in A$ with
  $Y \supset Y_n$.) Put $Y = \overline{\cup_{n = 1}^\infty Y_n}$. Note
  that $Y \in A$ as $c_0 \subset Y$ and $Y$ is separable. Observe that
  for $i = 0,1,2,3$ and all $n \ge k$ we have 
  \begin{align*}
    &\:\:\:\:\:\left|\: \tn x + y_i - z_k \tn  - \|x + y_i - z_k\|_{Y} \right|\\& \le
    \left|\: \tn x + y_i - z_k \tn  - \|x + y_i - z_k\|_{Y_n} \right| < \varepsilon_n,
  \end{align*}
  so $\tn x + y_i - z_k\tn  = \|x + y_i - z_k\|_Y$ as $\varepsilon_n \downarrow 0$.
  In particular, we have
  \begin{equation*}
    \|x - P_Y(x)\| \le \|x\|_Y = \|x + y_0 - z_0 \|_Y =\tn x + y_0 - z_0 \tn \leq 1.
  \end{equation*}
  Let $z = P_Y(x)$. Choose $j$ such that $\|z - z_j\|_{c_0} = \tn z -
  z_j \tn < \varepsilon$.
  Then we have
  \begin{equation*}
    \tn x + y_i - z \tn \le \tn x + y_i - z_j \tn + \varepsilon
    = \|x + y_i - z_j\|_Y + \varepsilon,
  \end{equation*}
  and since
  \begin{align*}
    \| x + y_i - z_j \|_Y &= \max\{\|P_Y(x) + y_i - z_j\|,
    \|x - P_Y(x)\|\}\\
    &\le \max\{\|y_i\| + \|z - z_j\|, 1\} \le 1 + \varepsilon,
  \end{align*}
  we get $\tn x + y_i - z \tn \le 1 + 2\varepsilon$.

To see that $T(X)=1$ we will show that in the new norm
  $B(e_1,1) \cup B(-e_1,1)$ covers $B_{(X,\tn \cdot\tn)}$.

  Let $x \in B_{(X,\tn \cdot\tn)}$. We use the same trick as before
  and find a separable subspace such that
  \begin{equation*}
    \tn x + z \tn  = \|x + z \|_Y
  \end{equation*}
  for $z=0$, $z=e_1$ and $z=-e_1$.

  We get
  \begin{align*}
    \tn x \pm e_1 \tn &= \|x \pm e_1 \|_Y
    = \max(\|P_Y(x \pm e_1)\|,\|x \pm e_1 - P_Y(x \pm e_1)\|) \\
    &= \max(\|P_Y(x) \pm e_1\|,\|x - P_Y(x)\|)
    \le \max(\|P_Y(x) \pm e_1\|,1)
  \end{align*}
  We know that $P_Y(x) \in B_{c_0}(e_1,1) \cup B_{c_0}(-e_1,1)$,
  so the above calculation shows that
  $x \in B(e_1,1) \cup B(-e_1,1)$.
\end{proof}

\end{document}